\documentclass{amsart}
\usepackage{amsmath,amsthm,amssymb,amscd}

\usepackage{graphicx}
\usepackage{amsthm}

\newtheorem{thm}{Theorem}[section]
\newtheorem{prop}[thm]{Proposition}
\newtheorem{lem}[thm]{Lemma}

\theoremstyle{definition}

\theoremstyle{remark}

\theoremstyle{remark}

\newcommand{\Z}{\mathbb{Z}}
\newcommand{\R}{\mathbb{R}}

\begin{document}

\title{Some examples of non-tidy spaces}

\author{Takahiro Matsushita}
\address{Graduate School of Mathematical Sciences, The University of Tokyo, 3-8-1 Komaba, Meguro-ku, Tokyo, 153-8914 Japan}
\email{tmatsu@ms.u-tokyo.ac.jp}

\begin{abstract}
We construct a free $\Z_2$-manifold $X_n$ for a positive integer $n$ such that $w_1(X_n)^n \neq 0$, but there is no $\Z_2$-equivariant map from $S^2$ to $X_n$.
\end{abstract}

\maketitle

\section{Introduction and Main theorems}

First we fix terminologies and notations we use in this paper. The group acts on spaces from the right unless otherwise stated. Let $\Gamma$ denote a group. In this paper, the $\Gamma$-action on a space $X$ is said to be free if for any $x \in X$, there is a neighborhood $U$ of $x$ such that $U \gamma \cap U = \emptyset$ for any $\gamma \in \Gamma \setminus \{ e_\Gamma\}$. For a $\Gamma$-space $X$, we denote the orbit space of $X$ by $\overline{X}$. Let $x_0 \in X$. The image of $\pi_1(X,x_0)$ via the group homomorphism $\pi_1(X,x_0) \rightarrow \pi_1(\overline{X},\overline{x}_0)$ induced by the quotient $X \rightarrow \overline{X}$ is often written by $\pi_1(X,x_0)$ also, for simplicity. The coefficient of the singular cohomology is considered as $\Z_2$, the cyclic group with order 2.

We write $S^n_a$ for the $n$-dimensional sphere with the antipodal action. For a free $\Z_2$-space $X$, we put
$${\rm coind}(X) = \sup \{ n \geq 0 \; | \; \textrm{There is a $\Z_2 $-map from $S^n_a$ to $X$.}\},$$
$${\rm ind}(X) = \inf \{ n \geq 0 \; | \; \textrm{There is a $\Z_2 $-map from $X$ to $S^n_a$.}\}.$$
$$h(X) =  \sup \{ n \geq 0 \; | \; w_1 (X)^n \neq 0 \}$$
and call the coindex, the index, and the Stiefel-Whitney height of $X$ respectively\footnote{These terminologies are due to \cite{Mat}. Many different terminologies are used, see \cite{Sto} or \cite{Tan}.}, where $w_1(X) \in H^1(\overline{X})$ is the 1st Stiefel-Whitney class of the double cover $X \rightarrow \overline{X}$.  It is obvious that
$${\rm coind}(X) \leq h(X) \leq {\rm ind}(X)$$
for every free $\Z_2$-space $X$. In \cite{Mat}, $X$ is said to be {\it tidy} if ${\rm ind}(X) = {\rm coind}(X)$.

In this paper, we prove the following.

\begin{thm}
For a positive integer $n$, there is a free $\Z_2$-space $X_n$ such that $h(X) =n$ but ${\rm coind}(X_n) = 1$.
\end{thm}

The space $X_n$ is defined as follows. Let $S^n_b$ denote the $\Z_2$-$\Z_2$-space where its base space is the $n$-dimensional sphere $S^n \subset \R^{n+1}$, and the left and the right $\Z_2$-actions are defined by
$$\tau(x_0,\cdots,x_n) = (-x_0,x_1, \cdots, x_n),$$
$$(x_0,\cdots, x_n) \tau = (-x_0,\cdots ,-x_n),$$
where $\tau$ is the generator of $\Z_2$. Then we define $X_1 = S^1_a$, and $X_{k+1} = X_k \times_{\Z_2} S^1_b$.

In \cite{Sch}, Schultz proved that $h(X \times _{\Z_2} S^n_b) \geq h(X) + n$ for any free $\Z_2$-space $X$. We prove the equality holds, although this is not necessary to prove Theorem 1.1.

\begin{thm}
$h(X \times_{\Z_2} S^n_b) = h(X) + n$ for any free $\Z_2$-space $X$.
\end{thm}

As far as I know, there is no explicit example published whose difference between the Stiefel-Whitney height and the coindex is greater than 1. On the other hand, it is known that the difference between ${\rm ind}(X)$ and $h(X)$ can be arbitrarily large. Indeed, the odd dimensional real projective space $\R P^{2n-1}$ is such example by the result of Stolz \cite{Sto}, see also \cite{Tan}.

\section{Proofs}
First we prove Theorem 1.1. As is said in Section 1, Schultz proved that $h(X \times_{\Z_2} S^n_b) \geq h(X) + n$. So we have that $h(X_n) \geq n$. Since $\overline{X}_n$ is an $n$-dimensional manifold (or by Theorem 1.2), we have $h(X_n) = n$. So what we must show is that there is no $\Z_2$-equivariant map from $S^2_a$ to $X_n$. To prove this, we establish a criterion to show the non-existence of equivariant maps, using fundamental groups.

Let $\Gamma$ be a discrete group, and $X$ a path-connected free $\Gamma$-space. Let $x_0 \in X$ be a base point. Then by the covering space theory, we have an isomorphism
$$\Gamma \cong \pi_1(\overline{X},\overline{x}_0) / \pi_1(X,x_0).$$
Recall that this isomorphism is given as follows. For $\alpha \in \pi_1(\overline{X},\overline{x}_0)$, let $\varphi \in \alpha$ and let $\tilde{\varphi}$ denote the lift of $\varphi$ whose initial point is $x_0$. Then the terminal point of $\tilde{\varphi}$ is in the fiber over $\overline{x}_0$, so there is a unique $\Phi_X(\alpha) \in \Gamma$ such that $\varphi (1) = x_0 \Phi_X(\alpha)$. This $\Phi_X : \pi_1(\overline{X},\overline{x}_0) \rightarrow \Gamma$ is a group homomorphism, which is surjective since $X$ is path-connected, and its kernel is $\pi_1(X,x_0)$. Hence $\Phi_X$ induces the isomorphism
$$\overline{\Phi}_X : \pi_1(\overline{X},\overline{x}_0)/ \pi_1(X,x_0) \longrightarrow \Gamma.$$

Let $X$ and $Y$ be connected free $\Gamma$-spaces and $f: X \rightarrow Y$ a $\Gamma$-equivariant map. Let $x_0 \in X$ and put $y_0 = f(x_0)$. Since the diagram
$$\begin{CD}
\pi_1(X,x_0) @>{f_*}>> \pi_1(Y,y_0)\\
@V{q_{X*}}VV @VV{q_{Y*}}V\\
\pi_1(\overline{X},\overline{x}_0) @>{\overline{f}_*}>> \pi_1(\overline{Y},\overline{y}_0)
\end{CD}$$
is commutative, so we have a group homomorphism
$$\hat{f}_* :\pi_1(\overline{X}, \overline{x}_0)/\pi_1(X,x_0) \rightarrow \pi_1(\overline{Y},\overline{y}_0)/\pi_1(Y,y_0).$$
We write $\Psi(f):\Gamma \rightarrow \Gamma$ for the group homomorphism which commutes the following diagram.
$$\begin{CD}
\Gamma @>{\Psi(f)}>> \Gamma\\
@A{\overline{\Phi}_X}AA @AA{\overline{\Phi}_Y}A\\
\pi_1(\overline{X},\overline{x}_0)/\pi_1(X,x_0) @>{\hat{f}_*}>> \pi_1(\overline{Y},\overline{y}_0)/\pi_1(Y,y_0).
\end{CD}$$
Then we have the following.

\begin{prop}
The group homomorphism $\Psi(f): \Gamma \rightarrow \Gamma$ is the identity.
\end{prop}
\begin{proof}
Let $\alpha \in \pi_1(\overline{X},\overline{x}_0) / \pi_1(X,x_0)$ and let $\varphi$ be the loop of $(\overline{X},\overline{x}_0)$ which represents $\alpha$. Let $\tilde{\varphi}$ denote the lift of $\varphi$ whose initial point is $x_0$. Then $\tilde{\varphi}(1) = x_0 \overline{\Phi}_X(\alpha)$. Then we have
$y_0 \overline{\Phi}_X (\alpha) = f \circ \tilde{\varphi}(1) = y_0 \overline{\Phi}_Y(\hat{f}_* \alpha)$. Hence $\overline{\Phi}_X = \overline{\Phi}_Y \circ \hat{f}_*$.
\end{proof}

The situation we used here is the case $\Gamma = \Z_2$. Let $X$ be a path-connected free $\Z_2$-space and $x_0 \in X$, we say $\alpha \in \pi_1(\overline{X},\overline{x}_0)$ is said to be even if $\alpha \in \pi_1(X,x_0)$, and is said to be odd if $\alpha$ is not even. Then Proposition 2.1 asserts that for a $\Z_2$-equivariant map $f:X \rightarrow Y$, the group homomorphism $\overline{f}_* :\pi_1(\overline{X},\overline{x}_0) \rightarrow \pi_1(\overline{Y},\overline{y}_0)$ preserves the parity of $\pi_1(\overline{X},\overline{x}_0)$.

Let us start to the proof of Theorem 1.1.

\begin{lem}
The group $\pi_1(\overline{X}_n)$ has no non-trivial torsion elements.
\end{lem}
\begin{proof}
By the definition of $X_n$, $\overline{X}_n$  is the orbit space of a free and isometrical $\pi_1(\overline{X}_n)$-action on $\R^n$. Remark that for an affine map $A: \R^n \rightarrow \R^n$, $a_1,\cdots, a_m \in \R$ with $\sum_{i=1}^m a_i = 1$, and $x_1,\cdots, x_m \in \R^n$, we have
$$A (\sum_{i = 1}^m a_i x_i) = \sum_{i=1}^m a_i (Ax_i).$$
Let $\alpha \in \pi_1(\overline{X}_n)$ be a non-trivial torsion element and its order is denoted by $k$. Let $x \in \R^n$. Then the point
$$y = \frac{1}{k}\sum_{i=1}^k x \alpha^i \in \R^n$$
is fixed by $\alpha$. This is contradiction.
\end{proof}

Hence to prove Theorem 1.1, it is sufficient to prove the following.

\begin{thm}
Let $X$ be a path-connected free $\Z_2$-space. Then there is a $\Z_2$-equivariant map from $S^2_a$ to $X$ if and only if there is $\alpha \in \pi_1(\overline{X})$ such that $\alpha^2 = 1$ and $\alpha \not\in \pi_1(X)$.
\end{thm}
\begin{proof}
Suppose there is a $\Z_2$-map $f: S^2_a \rightarrow X$. Let $\beta$ denote the generator of $\pi_1(\overline{S}^2_a) \cong \Z_2$. Since $\beta$ is odd, $\overline{f}_*(\beta)$ is odd. Since $\overline{f}_*(\beta) \cdot \overline{f}_*(\beta) = \overline{f}_*(\beta \cdot \beta) = 1$, we have completed the ``only if'' part.

On the other hand, suppose $\alpha \in \pi_1(\overline{X},\overline{x}_0)$ which is odd and $\alpha^2 = 1$. Let $\varphi \in \alpha$, and let $\tilde{\varphi}$ denote the lift of $\varphi$ whose initial point is $x_0$. Then $\psi = (\tilde{\varphi} \tau)\cdot \tilde{\varphi}$ is a loop of $(X,x_0)$, and is null-homotopic since $q_{X*}[\psi] = \alpha^2 = 1$ and $q_{X*}$ is injective, where $q_X:X \rightarrow \overline{X}$ is the quotient map. We can regard $\psi$ is a $\Z_2$-map from $S^1_a$ to $X$, and hence we can extend it to a $\Z_2$-map $S^2_a \rightarrow X$. This completes the proof.
\end{proof}

Finally, we prove Theorem 1.2.

\vspace{2mm}
\noindent {\it Proof of Theorem 1.2.} It is sufficient to prove that $h(X \times_{\Z_2} S^{n+1}_b) = h(X \times_{\Z_2} S^n_b) +1$ for $n \geq 0$. The part ``$\geq$'' is proved by Schultz in \cite{Sch}, so we give the proof of the part ``$\leq$''. Put
$$A = \{ (x_0,\cdots, x_{n+1}) \in S^n_b \; | \; |x_{n+1}| \leq \frac{1}{2}\},$$
$$B = \{ (x_0,\cdots ,x_{n+1}) \in S^n_b \; | \; |x_{n+1}| \geq \frac{1}{2}\}.$$
These are $\Z_2$-$\Z_2$-closed subset of $S^{n+1}_b$. Put $X' = X \times_{\Z_2} S^{n+1}_b$, $A' = X \times_{\Z_2} A$, and $B' = X \times_{\Z_2} B$. Then the followings hold.
\begin{itemize}
\item[(1)] $A' \simeq_{\Z_2} X \times_{\Z_2} S^n_b$.
\item[(2)] $B' \simeq_{\Z_2} \overline{X} \sqcup \overline{X}$ with the involution exchanging each $\overline{X}$. Hence $w_1(B') = 0$.
\item[(3)] $X' = A' \cup B'$. 
\end{itemize}

Suppose $w_1(X \times_{\Z_2} S^n_b)^k = 0$. Then $w_1(A')^k = 0$. Then there is $\alpha \in H^k(X',A')$ which maps to $w_1(X')^k$ via $H^k(X',A') \rightarrow H^k(X')$. Similarly, there is $\beta \in H^1(X',B')$ which maps to $w_1(X')$ via $H^1(X',B') \rightarrow H^1(X')$. Then $\alpha \cup \beta \in H^{k+1}(X',A' \cup B') = 0$, and which maps to $w_1(X')^{k+1}$. Hence we have $w_1(X')^{k+1} =0$. This completes the proof. \qed

\vspace{2mm}
\noindent {\bf Acknowledgement.} This work was supported by the Program for Leading Graduate Schools, MEXT, Japan.

\end{document}